\documentclass[article,11pt,reqno]{amsart}
\usepackage[margin=1.2in]{geometry}

\usepackage{amssymb,latexsym,amsmath,epsfig,amsthm} 
\usepackage[english]{babel} 
\usepackage[utf8]{inputenc} 
\usepackage{textcomp}
\usepackage{caption}
\usepackage{pdfpages}
\usepackage{chngpage}
\usepackage[hidelinks]{hyperref}

\makeatletter

\renewcommand\section{\@startsection {section}{1}{\z@}
{-30pt \@plus -1ex \@minus -.2ex}
{2.3ex \@plus.2ex}
{\normalfont\normalsize\bfseries}}

\renewcommand\subsection{\@startsection{subsection}{2}{\z@}
{-3.25ex\@plus -1ex \@minus -.2ex}
{1.5ex \@plus .2ex}
{\normalfont\normalsize\bfseries}}

\renewcommand{\@seccntformat}[1]{\csname the#1\endcsname. }

\makeatother

\newcounter{thmcounter}
\numberwithin{thmcounter}{section}

\newtheorem{theorem}[thmcounter]{Theorem}
\newtheorem{definition}[thmcounter]{Definition}
\newtheorem{lemma}[thmcounter]{Lemma}
\newtheorem{corollary}[thmcounter]{Corollary}
\newtheorem{conjecture}[thmcounter]{Conjecture}

\usepackage[procnames]{listings}
\usepackage{color}
\hypersetup{
    colorlinks,
    citecolor=black,
    filecolor=black,
    linkcolor=black, 
    urlcolor=black
}

\definecolor{dkgreen}{rgb}{0,0.6,0}
\definecolor{gray}{rgb}{0.5,0.5,0.5}
\definecolor{mauve}{rgb}{0.58,0,0.82}

\newcommand{\N}{\mathbb{N}}
\newcommand{\Z}{\mathbb{Z}}

\newcommand{\champ}{
\mathcal{T}
}

\newcommand{\seq}{
\mathcal{S}
}

\DeclareMathOperator{\ord}{ord}

\title{ Powers in prime bases and a problem on central binomial coefficients}

\author{Sebastian Tim Holdum}
\address{Sebastian Tim Holdum,Niels Bohr Institute, University of Copenhagen, Denmark}
\email{sebastian.holdum@nbi.dk}

\author{Frederik Ravn Klausen}

\address{Frederik Ravn Klausen,Department of Mathematics, University of Copenhagen, Denmark}
\email{ tlk870@alumni.ku.dk}

\author{Peter Michael Reichstein Rasmussen}
\address{Peter Michael Reichstein Rasmussen, Department of Mathematics, University of Copenhagen, Denmark}
\email{nmq584@alumni.ku.dk}

\begin{document}

\maketitle

\centerline{\bf Abstract}
\noindent
It is an open problem whether $ \binom{2n}{n} $ is divisible by 4 or 9 for all $n>256$. In connection with this, we prove that for a fixed uneven $m$ the asymptotic density of $k$'s such that  $ m \nmid \binom{2^{k+1}}{2^{k}} $ is 0. 
To do so we examine numbers of the form $\alpha^{k}$ in base $p$, where $p$ is a prime and $(\alpha, p)=1$. 
For every $n$ and $a$ we find an upper bound on the number of $k$'s less than $a$ such that $(\alpha^{k})_p$ contains less than $n$ digits greater than $\frac{p}{2}$.
This is done by showing that every sequence of the form $\langle  \sigma_t, \dots, \sigma_1,\sigma_0 \rangle$, where $0\leq \sigma_i<p$ for $i\geq 1$ and $\sigma_0$ is in the residue class generated by $\alpha$ modulo $p$, occurs at specific places in the representation $(\alpha^k)_p$ as $k$ varies. 

\pagestyle{myheadings} 
\thispagestyle{empty} 
\baselineskip=12.875pt 
\vskip 30pt

\section{Introduction}
A well known conjecture by Erd\H{o}s states that the central binomial coefficient $\binom{2n}{n}$ is never squarefree for $n>4$. The problem was finally solved in 1996 by Granville and  Ramaré \cite{E-formodning}, but is still inspiring further investigation of the central binomial coefficients. One question left unanswered can be found in \emph{Concrete Mathematics} \cite{GKP} and is the following conjecture which is the starting point of this paper.
\begin{conjecture}\label{main}
The central binomial coefficient $\binom{2n}{n}$ is divisible by 4 or 9 for every $n>4$ except $n=64$ and $n=256$.
\end{conjecture}

Since $4$ divides $\binom{2n}{n}$ when $n$ is not a power of 2, we consider only binomial coefficients of the form $\binom{2^{k+1}}{2^k}$ in our study of the conjecture. By Kummer's theorem the greatest exponent of a prime $p$ dividing the central binomial coefficient $\binom{2n}{n}$ is equal to the number of carries as $n$ is added to itself in base $p$. Thus, to prove the conjecture it is sufficient to show that there are at least 2 carries when $2^k$ is added to itself in base 3 and $k>8$.

In relation to this, Erd\H{o}s conjectured in 1979 \cite{erdos79} that the base 3 representation of $ 2^k $ only omits the digit 2 for $ k = 0,2, 8 $, noting that no methods for attacking it seemed to exist.  

Methods for analysing the digits of powers of a number $\alpha$ in prime bases are scarce, and further developing such methods is what most of this paper will be concerned with.

Considering the periodicity of the base $p$ representation of $\alpha^k$, for a prime $p$ and $(p, \alpha)=1$, we find new patterns that allow us to bound the function 
\begin{equation*}
  \seq_p^n(a) = \#\left\{0\leq s<a \mid (\alpha^s)_p \text{ contains less than $n$ digits greater than $\frac{p}{2}$} \right\}.
\end{equation*}
Specifically we show that every sequence of the form $\langle  \sigma_t, \dots, \sigma_1,\sigma_0 \rangle $, where $0\leq \sigma_i<p$ for $i\geq 1$ and $\sigma_0$ is in the residue class generated by $\alpha$ modulo $p$, occurs at given places in the representation $(\alpha^k)_p$ as $k$ varies.

Interestingly, if $p$ is not a Wieferich prime base $\alpha$, it turns out that this system occurs on every digit of $(\alpha^k)_p$. 

We use the above observations to show that
\begin{equation} \label{seq_main}
  \seq_p^n(a) \leq 8\left( \log_p(a) \right)^{n-1}a^{\log_p\left( \frac{p+1}{2} \right)},
\end{equation}
and in special cases we improve results due to Narkiewicz \cite{Nark}, and Kennedy and Cooper \cite{generalisering}.

The bound (\ref{seq_main}) is  used to prove that for any odd $m\in \N$, the set of numbers $k$ such that $m\nmid\binom{2^{k+1}}{2^k}$ has asymptotic density 0, which in the case $m=9$ specifically addresses conjecture \ref{main}.
\\

Lastly, we have used computer experiments to improve a result due to Goetgheluck \cite{binom_check} which confirmed Conjecture \ref{main} for all $n \leq 2^{4.2\cdot10^7}$.
\begin{theorem}
The central binomial coefficient $\binom{2n}{n}$ is divisible by 4 or 9 for every $n$ such that $4<n\leq2^{10^{13}}$ except for $n=64$ and $n=256$.
\end{theorem}
\noindent
See the Appendix for source code.

\section{Large digits in prime bases}
In this section we explore the base $p$ representation of powers of an integer $\alpha$, where $p$ is a prime not dividing $\alpha$. We say that a digit
$n$ is ``small'' if $n<\frac{p}{2}$ and ``large'' otherwise. Further, $p$ will always denote an odd prime, and $\alpha>1$ an integer with $(\alpha, p)=1$.

The main goal of the section is to bound the following function in various ways.
\begin{definition}
Let $p$ be an odd prime and $a, n\in \N$. Fix $\alpha$ such that $p\nmid \alpha$. Then set
\begin{equation*}
  \seq_p^n(a) = \#\left\{0\leq s<a \mid (\alpha^s)_p \text{ contains $<n$ large digits} \right\}.
\end{equation*}
\end{definition}
\noindent
Bounding the $S_p^n$ is done by considering periodic properties of $\alpha^k$ in base $p$ as $k$ varies.

\subsection{Notation and definitions}
\begin{definition}
Let $p$ be a prime and $n, k\in \N$. We write $p^k \mid\mid n$ if $p^k\mid n$ and $p^{k+1}\nmid n$, i.e. if $k$ is the greatest exponent of $p$ dividing $n$.
\end{definition}

\begin{definition}
We define the following:
\begin{itemize}
\item $\delta = \{\alpha^k \mod p \mid k\in \Z\}$, i.e. $\delta$ is the set of residues generated by $\alpha$ modulo $p$.
\item $\theta = \#\{a\in \delta \mid 0\leq a < \frac{p}{2}\}$, i.e. $\theta$ is the number of small residues in $\delta$.
\item $\gamma = \ord_p(\alpha) = \vert\delta\vert$.
\end{itemize}

\end{definition}

\begin{definition}
Let $n\in \N_0$. We let $\Lambda_n$ denote the set of sequences of the form 
\begin{equation*}
  \langle \sigma_n, \sigma_{n-1}, \dots, \sigma_{1}, \sigma_0\rangle,
\end{equation*}
where $\sigma_0 \in \delta$ and $0\leq \sigma_i < p$ for $1\leq i \leq n$.
\end{definition}

\begin{definition}
Let $m\in \N$ be represented in base $p$ as $m=\sum_{i\geq 0}a_ip^i$, 
where $ 0\leq a_j<p$.  To pinpoint specific digits we make the following definitions:
 $a_k = (m)_p[k]$ and $\langle a_k,\cdots, a_l\rangle = (m)_p[k:l], k\geq l$.
\end{definition}

\subsection{Sequences}
We will now consider the representations $(\alpha^s)_p$ when $s$ varies to show how members of $\Lambda_k$ occur as subsequences of these representations. 

First, we need a couple of lemmas.
\begin{lemma}\label{dobMid}
	Let $p$ be an odd prime and $\alpha>1$ be given such that $(p, \alpha)=1$. Let further $p^t\mid\mid \alpha^{\gamma p^k}-1$   for some $t>0$ and $k\geq 0$. Then $p^{t+1}\mid\mid \alpha^{\gamma p^{k+1}}-1$.
\end{lemma}
\begin{proof}
	Let $\alpha^{\gamma p^k}=up^t+1$ with $(u, p)=1$. Then
	\begin{equation*}
		\alpha^{\gamma p^{k+1}} = \left( up^t+1\right)^p = 1+up^{t+1}+u^2p^{2t}\binom p2 + R,
	\end{equation*}
	where $R$ is divisible by $p^{3t}$ and thus divisible by $p^{t+2}$ since $t>0$. Further, $p\mid \binom p2$, so $p^{t+2}\mid u^2p^{2t}\binom p2$ and we get
	\begin{equation*}
		\alpha^{\gamma p^{k+1}} \equiv 1+up^{t+1} \pmod{p^{t+2}}
	\end{equation*}
	showing that $p^{t+1}\mid\mid \alpha^{\gamma p^{k+1}}-1$.
\end{proof}

\begin{lemma}\label{dobbelMid}
Let $p$ be an odd prime and $\alpha>1$ be given such that $(p, \alpha)=1$. Assume that $p^{\tau}\mid\mid \alpha^{\gamma}-1$. Then 
\begin{equation*}
	p^{\tau+k}\mid\mid \alpha^{\gamma p^k}-1 \text{ and }\ord_{p^{\tau+k}}(\alpha)=\gamma p^k
\end{equation*}
for every $k\geq 0$.
\end{lemma}
\begin{proof}
The first part follows easily by induction on $k$ using Lemma \ref{dobMid}.\\
For the second part note that 
\begin{equation*}
	\gamma=\ord_p(\alpha)\mid \ord_{p^{\tau+k}}(\alpha)\text{ and }\ord_{p^{\tau+k}}(\alpha)\mid \gamma p^k.
\end{equation*}
Thus, $\ord_{p^{\tau+k}}(\alpha)=\gamma p^r$ for some $r\leq k$. By the first part, we have $p^{\tau+k-1}\mid\mid \alpha^{\gamma p^{k-1}}-1$, so $p^{\tau+k}\nmid \alpha^{\gamma p^{k-1}}-1$ and we must have $\ord_{p^{\tau+k}}(\alpha)=\gamma p^k$.
\end{proof}
With these lemmas at hand we are ready to analyse the base $p$ representation $(\alpha^s)_p$. To do so, we use the following definition.
\begin{definition}
Let $a=\dots a_2a_1a_0$ be any integer represented by an infinite sequence $(a_i)_{i\in \N_0}$ in some base. Then we define
\begin{equation*}
	c_{\tau, k}(a) = \langle a_{\tau+k-1}, \dots ,a_{\tau+1} , a_{\tau},a_0\rangle.
\end{equation*}
\end{definition}
We make this definition since  our interest lies in the digits underlined here:
\begin{equation*}
\dots \underline{a_{\tau+k-1}\dots a_{\tau}}\dots a_1\underline{a_0},
\end{equation*}
because all the elements of $\Lambda_n$ will appear periodically as subsequences of $(\alpha^s)_p$ on these positions, when $s$ changes. This is captured in the main theorem of the section.
\begin{theorem}\label{ck}
	Let $p$ be an odd prime and $\alpha>1$ be given such that $(p, \alpha)=1$. Further, let $\tau>0$ be the integer satisfying $p^{\tau} \mid\mid \alpha^\gamma-1$. Then for any $k\geq 0$
	\begin{equation*}
		\left\{c_{\tau, k}((\alpha^b)_p) \mid 0\leq b < \gamma p^k \right\} = \Lambda_k.
	\end{equation*}
\end{theorem}
\begin{proof}
Let $T:= \left\{c_{\tau, k}((\alpha^b)_p) \mid 0\leq b < \gamma p^k \right\}$.

Clearly, $T\subseteq \Lambda_k$ since every member of $T$ is of the form 
\begin{equation*}
  \langle \sigma_k, \sigma_{k-1}, \dots, \sigma_{1}, \sigma_0 \rangle,
\end{equation*}
where $0\leq \sigma_i < p$ for $1\leq i \leq k$ and $\sigma_0 \in \delta$, because $(\alpha^b)_p[0]\in \delta$ for any $b\geq 0$.

We now prove $T=\Lambda_k$, by showing $\vert T\vert = \gamma p^k = \vert \Lambda_k\vert$, where the last equality already follows from the definition of $\Lambda_k$.

Since $p^{\tau}\mid\mid \alpha^\gamma-1$ both $(\alpha^b)_p[\tau-1:0]$ and $(\alpha^b)_p[0]$ are periodic with respect to $b$ with least period $\gamma$ and no repetitions in the period. This means that for $b, c\geq 0$ we have 
$(\alpha^b)_p[\tau-1:0] = (\alpha^c)_p[\tau-1:0]$ if and only if $(\alpha^b)_p[0] = (\alpha^c)_p[0]$.

Now, assume for contradiction that $c_{\tau, k}((\alpha^b)_p)=c_{\tau, k}((\alpha^c)_p)$ for some $0\leq b<c<\gamma p^k$. Since $(\alpha^b)_p[0] = (\alpha^c)_p[0]$ we have $(\alpha^b)_p[\tau-1:0] = (\alpha^c)_p[\tau-1:0]$, so $(\alpha^b)_p[\tau+k-1:0]=(\alpha^c)_p[\tau+k-1:0]$, i.e. $\alpha^b\equiv \alpha^c \pmod{p^{\tau+k}}$. 
Therefore, $p^{\tau+k}\mid \alpha^{b}(\alpha^{c-b}-1)$, but this means that $p^{\tau+k}\mid\alpha^{c-b}-1$ contradicting Lemma \ref{dobbelMid} since $0<c-b<\gamma p^k$. 

Thus, all the elements in the definition of $T$ are different, and $\vert T\vert = \gamma p^k$.
\end{proof}

\subsubsection{Wieferich primes}
The main result of the section has a curious corollary related to the Wieferich primes.
\begin{definition}
Let $p$ be a prime and $\alpha>1$ be given such that $(\alpha, p)=1$. Then $p$ is a  Wieferich prime base $\alpha$ if $ p^2 \mid \alpha^\gamma -1 $.
\end{definition}
Since numerics \cite{Wie-num} indicate that for any $\alpha>1$ the Wieferich primes base $\alpha$ are somewhat scarce, it is interesting that the following elegant property holds for any  $(p, \alpha)$ such that $p$ is not a Wieferich prime base $\alpha$.

\begin{corollary}
Let $p$ be a prime which is not a Wieferich prime base $\alpha$.
Then 
\begin{equation*}
	\left\{(\alpha^b)_p[k:0]\mid 0\leq b < \gamma p^k\right\} = \Lambda_k. 
\end{equation*}
\end{corollary}

\begin{proof}
Since $p$ is not a Wieferich prime base $\alpha$, we have $ p^1 \mid \mid \alpha^\gamma $. Noticing that $c_{1, k}(a)=a[k:0]$ the corollary follows directly from Theorem \ref{ck}. 
\end{proof}

Thus, $p$ not being a Wieferich prime base $\alpha$ implies that the first $k+1$ digits of $(\alpha^s)_p$ will form all sequences of $\Lambda_k$ periodically as $s$ varies.

\subsection{Bounds on $\seq_p^n$}
The findings of the previous section allow us to obtain various bounds on the function $\seq_p^n$. First we introduce a lemma, which is a step on the way to bounding $\seq_p^n$ for $n=1$.

\begin{lemma}\label{SsP1}
Let $s, t\geq 0$, $p$ be a prime, and $\gamma=\ord_p(\alpha)$. Then we have 	
\begin{equation*}
  \seq_p^1(s \gamma p^t) \leq s \theta \left( \frac{p+1}{2} \right)^t.
\end{equation*}
\end{lemma}
\begin{proof}
	The number of sequences of $\Lambda_{t}$ containing only small digits is $\theta \left( \frac{p+1}2 \right)^t$. Thus, by Theorem \ref{ck} there are at most $\theta \left( \frac{p+1}2 \right)^t$ integers $0\leq h<\gamma p^t$, such that $(\alpha^h)_p$ does not contain any large digits. Now, letting $p^{\tau}\mid\mid \alpha^\gamma-1$ we have by Lemma \ref{dobbelMid} that the last $\tau+t-1$ digits of $(\alpha^h)_p$ are periodic with respect to $h$ with least period $\gamma p^{t}$ and no repetition in the period. Thus,
	\begin{equation*}
		\Lambda_t = \left\{c_{\tau, t}((\alpha^b)_p) \mid 0\leq b < \gamma p^t \right\} = \left\{c_{\tau, t}((\alpha^b)_p) \mid r \gamma p^t \leq b < (r+1)\gamma p^t \right\}
	\end{equation*}
	for every $r\in \N_0$, and we can see that there are at most $\theta \left( \frac{p+1}2 \right)^t$ integers $r\gamma p^t\leq h<(r+1)\gamma p^t$ such that $(\alpha^h)_p$ does not contain any large digits.
	
	This yields
	\begin{equation*}
  \seq_p^1(s \gamma p^t) \leq s \theta \left( \frac{p+1}{2} \right)^t.
\end{equation*}
\end{proof}
Now, the following theorem improves a result by Narkiewicz \cite{Nark} by a constant factor.

\begin{theorem} \label{S31}
Let $\alpha \equiv 2 \pmod{3}$ in the definition of $\seq$. For every $a\in \N$ we have
\begin{equation*}
  \seq_3^1(a) \leq 1.3 a^{\log_3(2)}.
\end{equation*}
\end{theorem}
\begin{proof}
 The theorem obviously holds for $a=1$. Now consider an $a\geq 2$, and let $s, t$ be given such that $s \in \{1, 2 \}$ and
 \begin{equation*}
  s \cdot 2\cdot 3^t \leq a \leq (s +1)\cdot 2\cdot 3^t.
\end{equation*}
We now have
\begin{equation*}
  t \leq \log_3(a)-\log_3(2s),
\end{equation*}
and since $S_3^1$ clearly is weakly increasing and by Lemma \ref{SsP1}, we get
\begin{align*}
\seq_3^1(a) & \leq \seq_3^1\left( (s+1)\cdot 2\cdot 3^t \right)\\
            & \leq (s + 1)\cdot 2^t \\
            & \leq (s + 1)\cdot 2^{-\log_3(2s)}\cdot 2^{\log_3(a)}.
\end{align*}
For $s \in \{1, 2\}$ the constant $(s + 1)\cdot 2^{-\log_3(2s)}$ is maximised by $s = 1$, and so
\begin{align*}
\seq_3^1(a) &\leq 2\cdot 2^{-\log_3(2)}\cdot 2^{\log_3(a)}\\
            &\leq 1.3a^{\log_3(2)}.
\end{align*}
\end{proof}
\noindent
The function $\seq_m^1$ for $m >2$ is studied by R. E. Kennedy and C. Cooper \cite{generalisering}, and if we consider only the cases when $m$ is a prime, we get the following improvement of their results which replaces a factor increasing with $m$ with a constant.

\begin{theorem}\label{SP1}
Let $p$ be a prime and $\alpha$ arbitrary in the definition of $\seq$. Then for all $a\in \N$ 
\begin{equation*}
  \seq_p^1(a) \leq 4 a^{\log_p \left (\frac{p+1}{2} \right)}.
\end{equation*}
\end{theorem}
\begin{proof}
The theorem holds for $a< \gamma$ since $a<4 a^{\log_p \left (\frac{p+1}{2} \right)}$ for $a<p$.

\noindent
Now let $a\geq \gamma$ and $s, t$ be integers with $0<s<p$ such that
\begin{equation*}
  s \gamma p^t \leq a < (s+1)\gamma p^t.
\end{equation*}
Now
\begin{equation*}
  t \leq \log_p(a)-\log_p(s\gamma),
\end{equation*}
and letting $\mu = \log_p\left( \frac{p+1}{2} \right)$ we get, by Lemma \ref{SsP1},
\begin{align*}
\seq_p^1(a) & \leq \seq_p^1((s+1) \gamma  p^t)\\
            & \leq (s+1) \theta \left( \frac{p+1}{2} \right)^t\\
            & \leq (s+1) \theta\left( \frac{p+1}{2} \right)^{\log_p(a)-\log_p(s\gamma)}\\
            & =    (s+1) \theta \left( s\gamma \right)^{-\mu} a^\mu.
\end{align*}
Since $\theta \leq \gamma<p$ we get
\begin{align*}
  \seq_p^1(a) & \leq \frac{s+1}{s^\mu}\gamma^{1-\mu} a^\mu\\
              & \leq \frac{s+1}{s^\mu} p^{1-\mu} a^\mu \\
              & =   \frac{s+1}{s^\mu} \frac{2p}{p+1} a^\mu.
\end{align*}
Considering $\frac{s+1}{s^\mu}$ we see that $\frac{d}{ds}\left( \frac{s+1}{s^\mu} \right) = s^{-\mu-1}(s(1-\mu)-\mu)$, and thus $\frac{s+1}{s^\mu}$ is strictly decreasing for $s\in\left[ 1, \frac{\mu}{1-\mu} \right[$ and strictly increasing for $s\in\left]\frac{\mu}{1-\mu}, p \right]$ and consequently attains its maximum on $[1, p]$ either at 1 or $p$. Since $s=1, s=p$ both yield $\frac{1+1}{1^{\mu}} = \frac{p+1}{p^{\mu}}=2$, we get
\begin{equation*}
  \seq_p^1(a) \leq 4a^\mu.
\end{equation*}
\end{proof}

\noindent
Finally we generalise our observations regarding $\seq_p^n$.
\begin{lemma}\label{Spn-inter}
Let $s\geq 0$, $t\geq 1$, $p$ be a prime, and $\gamma = \ord_p(\alpha)$. Then we have
\begin{equation*}
  \seq_p^n(s \gamma p^t) \leq 2 s \gamma t^{n-1} \left( \frac{p+1}{2} \right)^t.
\end{equation*}
\end{lemma}
\begin{proof}
For $t=1$ the result is clear. Now, assume $t>1$.

First, we count the number of sequences $\eta\in\Lambda_{t}$ such that $\eta $ contains less than $n$ large elements. This is done by counting for each $i<n$ how many sequences $\eta\in \Lambda_t$ that contain exactly $i$ large elements. 

For each $i$ we split up into two cases:\\
\textbf{Case 1:} The last element of $\eta$ is large (which means $i>0$). This element can then be chosen in $\gamma-\theta$ ways, and there are $\binom{t}{i-1}\left( \frac{p-1}{2} \right)^{i-1}\left( \frac{p+1}{2} \right)^{t+1-i}$ ways to choose the remaining $t$ elements such that exactly $i-1$ of them are large.\\
\textbf{Case 2:} The last element of $\eta$ is small. This element can then be chosen in $\theta$ ways, and there are $\binom{t}{i}\left( \frac{p-1}{2} \right)^{i}\left( \frac{p+1}{2} \right)^{t-i}$ ways to choose the remaining $t$ elements such that exactly $i$ of them are large.

Thus, we can express the number of elements in $\Lambda_{t}$ containing less than $n$ large elements by
\begin{align*}
\sum_{i=1}^{n-1}(\gamma-\theta)\binom{t}{i-1}\left( \frac{p-1}{2} \right)^{i-1}&\left( \frac{p+1}{2} \right)^{t+1-i}+
\sum_{i=0}^{n-1}\theta\binom{t}{i}\left( \frac{p-1}{2} \right)^{i}\left( \frac{p+1}{2} \right)^{t-i}\\
&\leq \gamma \left( \frac{p+1}{2} \right)^t \text{ }\sum_{i=0}^{n-1} \binom{t}{i}\\
& \leq \gamma \left( \frac{p+1}{2} \right)^t \text{ }\sum_{i=0}^{n-1} t^i\\
& \leq 2 \gamma  t^{n-1} \left( \frac{p+1}{2} \right)^t,
\end{align*}
since $t>1$.

Now, as in the proof of Lemma \ref{SsP1} we can conclude by Theorem \ref{ck} and Lemma \ref{dobbelMid} that for every $r\in \N_0$  there are at most $ 2 \gamma  t^{n-1} \left( \frac{p+1}{2} \right)^t$ integers $r \gamma p^t \leq k < (r+1)\gamma p^t$ such that $(\alpha^k)_p$ contains less than $n$ large digits. Thus, we have 
\begin{equation*}
  \seq_p^n(s \gamma p^t) \leq 2 s \gamma t^{n-1} \left( \frac{p+1}{2} \right)^t.
\end{equation*}

\end{proof}

\begin{theorem} \label{SPN}
Let $p$ be a prime and $\alpha$ arbitrary in the definition of $\seq$. Then for all $a, n\in \N$, where $a\geq \gamma p$, we have
\begin{equation*}
  \seq_p^n(a) \leq 8\log_p(a)^{n-1}a^{\log_p \left (\frac{p+1}{2} \right)}.
\end{equation*}
\end{theorem}
\begin{proof}
Let $a\geq \gamma p$ be given, and $s, t$ be integers with $0<s<p$ and $t\geq 1$ such that
\begin{equation*}
  s \gamma p^t \leq a < (s+1)\gamma p^t.
\end{equation*}
Now
\begin{equation*}
  t \leq \log_p(a)-\log_p(s\gamma),
\end{equation*}
and letting $\mu = \log_p\left( \frac{p+1}{2} \right)$ we use Lemma \ref{Spn-inter} and the fact that $\frac{s+1}{s^\mu} \gamma^{1-\mu}\leq 4$ from the the proof of Theorem \ref{SP1} to get 
\begin{align*}
\seq_p^n(a) &\leq \seq_p^n\left( (s+1)\gamma p^t \right) \\
            &\leq 2 (s+1) \gamma t^{n-1} \left( \frac{p+1}{2} \right)^t\\
            &\leq 2 (s+1) \gamma \left( \log_p(a)-\log_p(s\gamma) \right)^{n-1} \left( \frac{p+1}{2} \right)^{\log_p(a)-\log_p(s\gamma)}\\
            &\leq 2 (s+1) \gamma (s\gamma)^{-\mu} \log_p(a)^{n-1} a^\mu\\
            &= 2 \frac{s+1}{s^{\mu}} \gamma^{1-\mu} \log_p(a)^{n-1} a^\mu\\
            &\leq 8\log_p(a)^{n-1}a^{\log_p\left( \frac{p+1}{2} \right)}.
\end{align*}

\end{proof}

\section{Application to central binomial coefficients}
This section will apply the bounds on $\seq$ to a generalisation of Conjecture \ref{main} in order to show that the set of numbers not satisfying the conjecture  restricted to the case $n=2^s$ has asymptotic density 0.

For this we need the following theorem by Kummer.
\begin{theorem}[Kummer \cite{Kummer}]
Let $n, m\geq 0$ and $p$ be a prime. Then the greatest exponent of $p$ dividing $\binom{n+m}{m}$ is equal to the number of carries, when $n$ is added to $m$ in base $p$.
\end{theorem}
\noindent
Further we define the following function:
\begin{definition}
Let $m\in\N$ be odd. Then we define
\begin{equation*}
  \champ_m(a) = \#\left\{0\leq s<a\, \left\vert\, m \nmid \binom{2^{s+1}}{2^s} \right.\right\}.
\end{equation*}
\end{definition}
\noindent
\noindent
It is clear, that to show Conjecture \ref{main} we would have to bound $\champ_9$ by $\champ_9(a)\leq 5$ for all $a$. Instead we can get a partial result by connecting $\champ$ and $\seq$ in the following way:
\begin{lemma}
Let $a, n\in \N$, $\alpha=2$ in the definition of $\seq$, and $p$ be an odd prime. Then
\begin{equation*}
  \champ_{p^n}(a) \leq \seq_p^n(a).
\end{equation*}
\end{lemma}
\begin{proof}
Adding $2^s$ to itself in base $p$ will yield at least one carry for every large digit in $(2^s)_p$. Thus, by Kummer's theorem, we must have
\begin{equation*}
  \champ_{p^n}(a) \leq \seq_p^n(a).
\end{equation*}
\end{proof}
\noindent
With this at hand it is possible to give an asymptotic upper bound on $\champ_m$ for every odd $m$.
\begin{theorem}\label{density}
Let $m>1$ be odd and let $p$ be the greatest prime dividing $m$. Then 

\begin{equation*}
  \champ_{m}(a) = o\left( a^{\log_{p}\left( \frac{p+1}{2} \right) +\epsilon } \right)
\end{equation*}
for any $\epsilon>0$.
\end{theorem}
\begin{proof}
Assume $m$ has prime factorisation $m=p_1^{\beta_1}p_2^{\beta_2}\cdots p_k^{\beta_k}$ with $p_1<p_2<\dots<p_k$. 
Then $\seq_{p_i}^{\beta_i}(a)=O\left(\log_{p_k}(a)^{\beta_{k}-1}a^{\log_{p_k}\left( \frac{p_k+1}{2} \right)}\right)$ for all $1\leq i\leq k$, since $p_i\leq p_k$, and thus, 
\begin{align*}
\champ_m(a) &\leq \sum_{i=1}^k \seq_{p_i}^{\beta_i}(a)\\
            &= O\left(\log_{p_k}(a)^{\beta_{k}-1}a^{\log_{p_k}\left( \frac{p_k+1}{2} \right)}\right) \\
            &= o\left( a^{\log_{p_k}\left( \frac{p_k+1}{2} \right) +\epsilon } \right)
\end{align*}
for any $\epsilon>0$.
\end{proof}
\noindent
Although we still cannot give a definite answer to Conjecture \ref{main}, we do get the following corollary.
\begin{corollary}
	For every odd $m$ the set of integers $s$ such that $m\nmid \binom{2^{s+1}}{2^s}$ has asymptotic density 0.
\end{corollary}
\begin{proof}
	By Theorem \ref{density} we have $\champ_m(a) = o\left( a \right)$.
\end{proof}
\noindent
Since the case $m=9$ is not special in this corollary, it seems natural to pose the following conjecture, which strengthens Conjecture \ref{main}.
\begin{conjecture}
	For every odd $m$ there is an $N\in \N$ such that $m\mid \binom{2^{k+1}}{2^k}$ for every $k\geq N$.
\end{conjecture}
\noindent
It seems by Theorem \ref{ck} and by computer heuristics that the digits of $(2^s)_p$ are uniformly distributed for large $s$ in the sense that for any $0\leq a<p$ most digits in the representation have probability roughly $1/p$ of being $a$.

Assuming such a random distribution of the digits in the representation and considering computer experiments on a selection of primes $p<200$ has lead to the following conjecture. 
\begin{conjecture}\label{UF}
	For an odd prime, $p$, let $\epsilon_p(a)$ be the function satisfying $p^{\epsilon_p(a)}\mid\mid a$ for every $a$. Then 
		\begin{equation*}
		\epsilon_p\left( \binom{2^{k+1}}{2^k} \right) = \frac{\log(2)}{2\log(p)}\cdot k+O(\sqrt{k}).
	\end{equation*}
\end{conjecture}

\vskip 30pt
\paragraph{Acknowledgement}
The authors wish to thank prof. S\o ren Eilers for his helpful guidance and suggestions in the writing process; prof. Carl Pomerance for his encouragement and support; and the anonymous referee for his/her suggestions that significantly improved some of the proofs of this paper.

\begin{appendix}
\section{Source code}
    The following code checks that the central binomial coefficient $\binom{2n}{n}$ is divisible by 4 or 9 for every $n$ such that $4<n\leq2^{10^{13}}$ except for $n=64$ and $n=256$. 

The Java-code checks the first 35 digits of the base 3 representation of $2^k$ for every $k$ such that $0<k<10^{13}$. Every $k$ such that the first 35 digits of $2^k$ do not contain two 2's is written to a file containing special cases. These cases are then checked individually by the Python-code.

\newpage
\thispagestyle{empty}
\subsection*{JAVA source}
\addcontentsline{toc}{section}{JAVA-kode}
\begin{lstlisting}
import java.io.FileWriter;
import java.io.IOException;
import java.io.File;

class NewSearcher {
    private static int[] number = new int[35];
    private static int size = 0;
    private static final int MAX_SIZE = 35;
    
    private static final String ERROR_FILE = "Check_needed.txt";
    
    public static void main(String[] args) {
        deleteFile(ERROR_FILE);
        
        addNum(1);
        
        for (int a=0; a<10000000; a++) {
            for (int b=0; b<1000000; b++) {
                if (doubleIt()) {
                    String output = String.format("%d%06d", a, b);
                    System.out.println(output);
                    writeNumberToFile(ERROR_FILE, output);
                }
            }
        }
    }
    
    private static void addNum(int num) {
        if (size < MAX_SIZE) {
            number[size] = num;
            size ++;
        }
    }
    
    public static boolean doubleIt() {
        int totalCarry = 0;
        int carry = 0;
        int i=0;
        
        while (totalCarry < 2 && i<size) {
            int res = (number[i]*2 + carry);
            carry = (res>=3) ? 1 : 0;
            number[i] = (res % 3);
            if (carry==1) totalCarry ++;
            i++;
        }

        while (i<size) {
            int res = (number[i]*2 + carry);
            carry = (res>=3) ? 1 : 0;
            number[i] = (res % 3);
            i++;
        }
        
        if (carry == 1) {
            addNum(1);
        }
        return (totalCarry<2);
    }
    
    public static void writeNumberToFile(String filename, String number) {
        try
        {
            FileWriter fw = new FileWriter(filename, true);
            fw.write(number + "\r\n");
            fw.close();
        }
        catch(IOException e)
        {
            System.out.println("IOException: " + e.getMessage());
        }
    }
    
    public static void deleteFile(String filename) {
        try {
            File toDelete = new File(filename);
            toDelete.delete();
        } catch (Exception e) {
            
        }
    }
}
\end{lstlisting}

\definecolor{keywords}{RGB}{255,0,90}
\definecolor{comments}{RGB}{0,0,113}
\definecolor{red}{RGB}{160,0,0}
\definecolor{green}{RGB}{0,150,0}

\lstset{language=Python, 
        basicstyle=\ttfamily\small, 
        keywordstyle=\color{keywords},
        commentstyle=\color{comments},
        stringstyle=\color{red},
        showstringspaces=false,
        identifierstyle=\color{green},
        procnamekeys={def,class}}
\newpage
\thispagestyle{empty}
\subsection*{Python source}
\addcontentsline{toc}{section}{Python-kode}
\begin{lstlisting}

def mod(n, md):
    if n < 10:
        return 2**n%md
    
    return 2**(n%2)*mod(n/2, md)**2%md

def checkCarry(n):
    tmp = n
    count = 0
    while tmp and count<2:
        if tmp%3 == 2:
            count += 1
        tmp /= 3
        
    return count<2
        
fil = file("Check_needed.txt", "r")

nls = []

while True:
    try:
        next = int(fil.readline())
        if checkCarry(mod(next, 3**50)):
            nls.append(next)
    except ValueError:
        break

for i in nls:
    if checkCarry(mod(i, 3**80)):
        print i
\end{lstlisting}

\end{appendix}
\end{document}